\UseRawInputEncoding
\documentclass[12pt,draft,reqno]{amsart}

\usepackage{array,latexsym}
\usepackage{amsmath,amssymb,amscd,amsfonts,amsthm}

\usepackage{paralist}

\usepackage{mathrsfs}
\usepackage{dsfont}
\usepackage{bbold}
\usepackage{mathbbol}
\usepackage[all]{xy}
\usepackage{enumerate}

\makeatletter
\renewcommand*\env@matrix[1][\arraystretch]{%
  \edef\arraystretch{#1}%
  \hskip -\arraycolsep
  \let\@ifnextchar\new@ifnextchar
  \array{*\c@MaxMatrixCols c}}
\makeatother

{\catcode`\@=11
\gdef\n@te#1#2{\leavevmode\vadjust{%
 {\setbox\z@\hbox to\z@{\strut#1}%
  \setbox\z@\hbox{\raise\dp\strutbox\box\z@}\ht\z@=\z@\dp\z@=\z@%
  #2\box\z@}}}
\gdef\leftnote#1{\n@te{\hss#1\quad}{}}
\gdef\rightnote#1{\n@te{\quad\kern-\leftskip#1\hss}{\moveright\hsize}}
\gdef\?{\FN@\qumark}
\gdef\qumark{\ifx\next"\DN@"##1"{\leftnote{\rm##1}}\else
 \DN@{\leftnote{\rm??}}\fi{\rm??}\next@}}
\DeclareFontFamily{OT1}{wncyr}{\hyphenchar\font45
}
\DeclareFontShape{OT1}{wncyr}{m}{n}{%
   <5> <6> <7> <8> <9> gen * wncyr
   <10> <10.95> <12> <14.4> <17.28> <20.74>  <24.88>wncyr10}{}
\DeclareFontShape{OT1}{wncyr}{m}{it}{%
   <5> <6> <7> <8> <9> gen * wncyi
   <10> <10.95> <12> <14.4> <17.28> <20.74> <24.88> wncyi10}{}
\DeclareFontShape{OT1}{wncyr}{m}{sc}{%
   <5> <6> <7> <8> <9> <10> <10.95> <12> <14.4>
   <17.28> <20.74> <24.88>wncysc10}{}
\DeclareFontShape{OT1}{wncyr}{b}{n}{%
   <5> <6> <7> <8> <9> gen * wncyb
   <10> <10.95> <12> <14.4> <17.28> <20.74> <24.88>wncyb10}{}
\input cyracc.def

\DeclareMathSizes{7}{7}{5}{3}

\theoremstyle{plain}

\newtheorem{theorem}{Theorem}

\newtheorem{lemma}{Lemma}

\theoremstyle{definition}

\newtheorem{definition}{Definition}

\newtheorem{nothing*}[theorem]{}
\newtheorem{subnothing*}[sub]{}
\newtheorem{example}{Example}

\theoremstyle{remark}

\newtheorem{remark}{Remark}

\newcommand{\cc}{\raise .4pt \hbox{{$\scriptstyle{\bullet}$}}}

\begin{document}

\title[Rationality of  adjoint orbits]{Rationality of  adjoint orbits}

\author[Vladimir  L. Popov]{Vladimir  L. Popov${}^*$}
\thanks{${}^*$\,Steklov Mathematical Institute,
Russian Academy of Sciences, Gub\-kina 8,
Moscow 119991, Russia. {\it E-mail address}: popovvl@mi-ras.ru}


\begin{abstract}
We prove that every
 orbit of the adjoint representation of  any connected
reductive algebraic group $G$ is a rational algeb\-raic variety. For complex simply connected semisimple $G$, this implies rationality of
  affine Hamiltonian $G$-varieties (which we classify).
\end{abstract}

\maketitle

\section{\bf 1.\;Introduction}

Let $G$ be a connected affine algebraic group and let $H$ be its closed subgroup.
Whether the algebraic variety $G/H$ is rational is a well-known old problem closely related to the rationality problem of  inva\-riant fields of linear representations of algebraic groups (see
\cite[1.5]{Po94}, \cite[Thm.\,1, Cor.\,2]{Po13}).

As is shown  in \cite[p.\,298]{Po11},
 \cite [Thm.\;2]{Po13}, \cite[Rem. 1.5.9]{Po94}, for some $G$ and {\it finite} $H$,
 the variety $G/H$ is nonrational (and even not stably rational).
  However, the existence of nonratio\-nal varie\-ties
  $G/H$ with a {\it connected} group $H$ is still an
  intriguing open
  problem.

At the same time, for many pairs $(G, H)$ with connected group $H$ either rationality
or stable rationality of the variety $G/H$ is proved; for instance, $G/H$ is rational whenever $\dim(G/H)\leqslant 10$  (see \cite{CZ15}).

The conference talk \cite{Ba151}
served for me as an impetus
to explore rationality of
orbits of
the
adjoint representations of
connected reduc\-tive groups. Searching for some special rational
coordi\-nates on the adjoint orbits of
${\rm GL}_n(\mathbb C)$, ${\rm SO}_n(\mathbb C)$, ${\rm Sp }_n(\mathbb C)$,
the author of
\cite{Ba151}
proved, as a byproduct,
rationality of the majority of  these orbits.
He uses the ``method of the canonical orbit parameterization'' that
dates back to the work of I.\,M.\,Gelfand and M.\,I.\,Nai\-mark
on unitary representations of classical groups (1950). The para\-meterization of (co)\-adjoint orbits was of inte\-rest to many researchers because of
its connection with the problems of the
theory of
integr\-able systems
(see introduction
and related references
in \cite{Ba16}).

In the present paper, is proved the following theorem announced in \cite{Po16}, which yields infinitely many new examples of rational varieties of the form $G/H$.

\begin{theorem}\label{ratadj} Let $G$ be a connected reductive algebraic group.
Every
$G$-orbit of the adjoint representation of $G$
is a rational algebraic variety.
\end{theorem}

The  $G$-orbits in Theorem \ref{ratadj} are exactly the varieties
$G/H$, where $H$ is the $G$-centralizer $C_G(x)$ of an element $x$ of the Lie algebra
$$\mathfrak g:={\rm Lie}(G).$$
The groups $C_G(x)$
 have been
thoroughly studied (see \cite{CM93}, \cite{Hu95}, \cite{SS70}).
Among them there are both connected and disconnected gro\-ups. Their dimensions are
not less than
$r:={\rm rk}(G)$.

Note
that Theorem \ref{ratadj} establishes a specific property of the ad\-joint representation:
by Remark\;\ref{rem} below, there exist representations of
some $G$ not all of whose orbits are rational algebraic varieties.

Theorem \ref{ratadj} is
applied
in the proof
of
the following Theorem \ref{symplhom}, which
concerns
the classification and properties of
Hamiltonian $G$-varieties.
In it,
the following notation is used:
\begin{align*}
{\mathscr H}\;&\mbox{is the set of  isomor\-phism classes of affine Hamiltonian $G$-varieties},\\[-.7mm]
{\mathscr S}\;&\mbox{is the set of $G$-orbits of nonzero semisimple elements of\;$\mathfrak g$.}
\end{align*}

\begin{theorem}\label{symplhom} If $G$ is a simply connected
complex semisimple algebraic group and $X$ is an affine Hamiltonian $G$-variety,
then the following holds.
\begin{enumerate}[\hskip 4.2mm\rm (a)]
\item\label{or}
$X$ is isomorphic to a
unique
$G$-orbit ${\mathcal O}_X\in \mathscr S$ endowed with the stan\-dard
structure of
a Hamiltonian $G$-variety \textup(see {\rm \cite[5.2]{Ko70}}\textup).

\item\label{bije} The map $X\mapsto {\mathcal O}_X$ yields a bijection
\begin{equation}\label{bij}
\mathscr{H}\to\mathscr{C}.
  \end{equation}
    \item\label{scrv} $X$ is a simply connected rational variety.
    \item\label{stab} The $G$-stabilizer of any point of $X$ is a Levi subgroup of a proper parabolic subgroup of $G$.
    \item\label{Levs} Every Levi subgroup of every proper parabolic subgroup of $G$ is
    the  $G$-stabi\-lizer of a point of
    some affine Hamiltonian $G$-variety.
\end{enumerate}
\end{theorem}

Recall  that for the adjoint action of $G$ on $\mathfrak g$,
the catego\-ri\-cal quotient $\mathfrak g/\!\!/G$  is isomorphic to the affine space ${\mathbb A}^r$, and every fiber of the
quotient morphism
$G\to \mathfrak g/\!\!/G$
contains a unique
$G$-orbit
from $\mathscr S$ (see \cite{Ko63}, \cite[8.5]{PV94}).  Combined with the existence of bijection \eqref{bij}, this yields
a paramitrization of $\mathscr H$ by
${\mathbb A}^r\setminus \{0\}$.

The proofs of Theorems \ref{ratadj} and \ref{symplhom}
are given in Section 3.

\vskip 2mm

\noindent {\it Conventions and notation.} Our basic reference for algebraic groups and algebraic geometry is \cite{B91}
and
we follow the conventions
therein.
Unless otherwise stated, all algebraic groups and algebraic va\-rieties are taken over an algebraically closed
field $k$ whose characteristic is
not a bad prime for reductive $G$ (see \cite[Chap.\,I, Def.\,4.1]{SS70}).

\vskip 1.5mm
We use the following notation:

\vskip 1mm

$C_G(M)$ is the $G$-centralizer of a subset $M$ of $\mathfrak g$ or $G$.

${\rm Rad}_u(Q)$ is the unipotent radical of an affine algebraic group $Q$.

$\mathbb A^n$ is the $n$-dimensional affine space.

\section{\bf 2. Birational complements}

\begin{definition}\label{sect}  Let $G$ and $H$ be as in Section $1$.
A sequence $S_1,\ldots, S_m$ of locally closed subsets of $G$
is called a
{\it birational complement to $H$ in $G$} if the morphism
\begin{equation}\label{bc}
\lambda\colon S_1\times\cdots \times S_m\times H\to G,\quad (s_1,\ldots, s_m, h)\mapsto s_1\cdots s_mh
\end{equation}
is an open embedding.
\end{definition}

\begin{remark}
One can show that this notion is order-sensitive, i.e.,  re\-shaf\-fling the terms of a sequence that is a birational comple\-ment to $H$ in $G$,
one obtains
a sequence that, generally speaking,  is not a birational complement to $H$ in $G$.
\end{remark}

If $S_1,\ldots, S_m$ is a birational comp\-le\-ment to $H$ in $G$, then
Defini\-tion\;\ref{sect} implies, in view of the connectedness of $G$, that
$H$ is connected as well. It also implies that  open embedding \eqref{bc} is $H$-equivariant with respect to the action of $H$ on $G$ by right translations and on $S_1\times\cdots\times S_m\times H$ by right translations of the last factor. Thus, the image of open embedding \eqref{bc} is an $H$-invariant open subset of $G$ that is  $H$-equivariantly isomorphic to
$S_1\times\cdots \times S_m\times H$.

\begin{example}\label{ex1} Let $G$ be a semidirect product $A\ltimes B$ of closed subgroups $A$ and $B$.
Then the one-term sequence $A$ (resp. $B$) is a birational complement  to $B$ (resp. $A$) in $G$.
For instance, one can take $A$ and $B$ to be respectively a Levi subgroup
of $G$ and the unipotent radical ${\rm Rad}_u (G)$.
\end{example}

\begin{example}\label{ex2}
 Let $G$ be a reductive algebraic group and let $P$ be a parabolic subgroup of $G$.
 Let
 $P^-$ be the parabolic subgroup of $G$ opposite to $P$. Then the one-term sequence ${\rm Rad}_u(P^-)$ is a birational complement in $G$ to $P$ (see \cite[Prop.\;14.21(iii)]{B91}).
\end{example}

\eject

\begin{lemma}\label{propcomp} Let $G$ and $H$ as in Section $1$ and let $Q$ be a closed subgroup of $H$.
Let $S_1,\ldots , S_m$ be a birational complement to $H$ in $G$.
\begin{enumerate}[\hskip 4.2mm\rm(a)]
\item\label{merge} If
$Z_1,\ldots, Z_n$ is a birational complement to $Q$ in $H$, then
$$S_1,\ldots , S_m, Z_1,\ldots, Z_n$$ is a birational comp\-lement to $Q$ in $G$.
\item\label{osu3} The variety $G/Q$ contains an open subset isomorphic to
\begin{equation}\label{SHQ}
S_1\times\cdots\times S_m\times (H/Q).
\end{equation}
\end{enumerate}
\end{lemma}

\begin{proof}
\eqref{merge} Let $X:=S_1\times\cdots\times S_m$ and $Y:=Z_1\times\cdots\times Z_n$. By Definition\,\ref{sect},
\begin{equation*}
\mu\colon Y\times Q=Z_1\times\cdots\times Z_n\times Q\to H,\quad (z_1,\ldots, z_n, q)\mapsto z_1\cdots z_nq
\end{equation*}
is open embedding, therefore,
$\nu:=\colon {\rm id}_X\times \mu\colon X\times (Y\times Q)\to X\times H$
is an open embedding. Since $\lambda$ (see  \eqref{bc}) is also an open embedding,
this implies that the morphism
\begin{gather*}
\lambda\circ\nu\colon X\times (Y\times Q)=S_1\times\cdots \times S_m\times Z_1\times \cdots\times Z_n\to G,\\[-.5mm]
(s_1,\ldots, s_m, z_1,\ldots, z_n, q)\mapsto s_1\cdots s_m z_1\cdots z_n q
\end{gather*}
is an open embedding as well. This proves \eqref{merge}.

\eqref{osu3}  As noted above, $G$ contains an $H$-invariant open subset $U$ that is $H$-equivariantly isomorphic to $S_1\times\cdots \times S_m\times H$. Therefore, by \cite[II, Thm.\,6.8 and Cor.\,6.6]{B91}, a geometric quotient $U/Q$ exists and is isomorphic to variety \eqref{SHQ}. On the other hand, $U/Q$ is isomorphic to an open subset of $G/Q$ because the canonical morphism $G\to G/Q$ is open (see \cite[II, 6.1]{B91}).
This proves \eqref{osu3}.
\end{proof}

\begin{example}\label{ex4} Taking $H=Q$ in Lemma \ref{propcomp}(b) yields that $G/H$ contains an $H$-invariant open subset isomorphic to $S_1\times\cdots \times S_m$. For instance, $G/P$ in Example \ref{ex2} contains an open subset isomorphic to the under\-lying variety of ${\rm Rad}_u (P^-)$, i.e.,
to an affine space (see \cite[IV, 14.4, Rem.]{B91}); whence, $G/P$ is rational.
\end{example}

\begin{example}\label{ex3}
Let $G$ be a reductive group, let $P$ be its a parabolic subgroup and
let $L$ be a Levi subgroup of $P$. Then by Lemma \ref{propcomp} and Examples \ref{ex1}, \ref{ex2},
the two-term sequence ${\rm Rad}_u(P^-)$, ${\rm Rad}_u(P)$ is a birational complement to $L$ in $G$ and $G/L$ contains an open subset isomorphic to the underlying variety of ${\rm Rad}_u(P^-)\times {\rm Rad}_u(P)$, i.e., to an affine space; whence $G/L$ is a
rational variety.
\end{example}

\begin{example}\label{ex5} Let $B$ be a Borel subgroup of $G$ and let
$Q$ be a closed subgroup of $B$.
In view of  Lemma \ref{propcomp} and Example \ref{ex2}, the variety $G/Q$ contains an open subset isomorphic to $\mathbb A^d\times (B/Q)$. By \cite[Thm.\,5]{Ro62},
the variety
$B/Q$ is isomorphic to $\mathbb A^s\times (\mathbb A^1\setminus \{0\})^t$ for some $s, t$. Whence
$G/Q$ is a rational variety
(this
statement
is
Theorem 2.9 of \cite{CZ15}).
Since every
connected solvable subgroup
of $G$ lies in a Borel subgroup of $G$, this implies that the variety $G/H$ is rational
if $H$ is connected solvable.
\end{example}

\section{\bf 3. Proofs of Theorems \ref{ratadj} and \ref{symplhom}}\label{secti}

 In the proof of Theorems \ref{ratadj} and \ref{symplhom},  we shall use the following

\begin{theorem}\label{red-stab} \

$\cc$ Let $G$ be a connected reductive algebraic group and let $x$ be an element of
$\mathfrak g$.
The following pro\-per\-ties
are equivalent:
\begin{enumerate}[\hskip 4.2mm\rm(a)]
\item\label{ss} $x$ is semisimple;
\item\label{af} the $G$-orbit $\mathcal O$ of $x$ is an affine variety;
\item\label{re}
$C_G(x)$ is reductive;
\item\label{cl}
$\mathcal O$
is a closed subset of $\mathfrak g$.
\end{enumerate}

$\cc$ Let $G$ be a simply connected semisimple algebraic group. Then
pro\-per\-ties
{\rm \eqref{ss}--\eqref{cl}} are equivalent to
the property
\begin{enumerate}[\hskip 4.2mm {\rm(e)}]
\item \label{ls}
$C_G(x)$ is a Levi subgroup of a parabolic subgroup of $G$.
\end{enumerate}
For any parabolic subgroup $P$ of $G$ and any Levi subgroup $L$  of $P$,
there is a semisimple element $s\in \mathfrak g$ such that $C_G(s)=L$.
\end{theorem}

\begin{proof}
 \

\eqref{ss}$\Leftrightarrow$\eqref{cl} is well-known (see \cite[Rem.\,11]{Ko63}, \cite[8.5]{PV94}).

\eqref{cl}$\Rightarrow$\eqref{af} is clear.

\eqref{af}$\Leftrightarrow$\eqref{re} follows from Matsushima's criterion  (see \cite[Thm.\,4.17]{PV94}, \cite{Ri75}).

\eqref{af}$\Rightarrow$\eqref{cl} Let  $\mathcal O$ be affine and
let $\overline{\mathcal O}$ be the closure of $\mathcal O$ in $\mathfrak g$. The set $X:=\overline{\mathcal O}\setminus \mathcal O$ is closed in $\mathfrak g$  (see \cite[1.3]{PV94}). Arguing on the contrary, assume that
$X\neq \varnothing.$ Since $\mathcal O$ is affine, this yields
${\rm codim}_{\overline{\mathcal O}}(X)=1$
(see \cite{Go69}). On the other hand, ${\rm codim}_{\overline{\mathcal O}}(X)\geqslant 2$ by \cite[Cor.\,1 of Thm.\,3]{Ko63}.
This contradiction shows that
$\overline{\mathcal O}= \mathcal O$.

\eqref{ss}$\Rightarrow$(e)
Let $x$ be a semisimple element. By
\cite[Cor.\,3.8]{S75},
there is a torus $S$ in $G$
such that $x\in \mathfrak s:={\rm Lie}(G)$ and
\begin{equation}\label{ZzZ}
C_G(x)=C_G(\mathfrak s).
\end{equation}
By \cite[Prop.\,3.4.7]{DM20}, there is a parabolic subgroup
$P$
of $G$ such that
$C_G(S)$ is a Levi subgroups of $P$. Since
\begin{equation}\label{Ss}
C_G(S)=C_G(\mathfrak s)
\end{equation}
(see \cite[Thm. 24.4.8(ii)]{TY05}), from
\eqref{ZzZ} and \eqref{Ss} we infer that
$C_G(x)$ is a Levi subgroups of $P$.

(e)$\Rightarrow$\eqref{re} is clear.

To prove the last statement of Theorem \ref{red-stab}, note that since $L$ is a Levi subgroup of a parabolic subgroup of $G$,
there is a torus $S$ in $G$ such that
\begin{equation}\label{L}
L=C_G(S)
\end{equation}
(see \cite[Prop.\,3.4.6]{DM20}). We identify $S$ with  $(k^*)^d$ by means an isomor\-phism between them.
Let $\mathfrak s=k^d
$ be the Lie algebra of $S$. If  $z=(z_i)\in \mathfrak s$, let $R_z:=\{(m_i)\in\mathbb Z^d\mid \sum_i m_iz_i=0\}$. Then
the minimal algebraic sub\-al\-gebra $\mathfrak a(z)$ of  $\mathfrak s$ containing $z$
is $\{(s_i)\in \mathfrak s\mid \sum_im_is_i=0\;\;\mbox{for all $(m_i)\in R_z$}\}$
(see \cite[II,\,7.3(2)]{B91}).
Since the degree of $k$ over its  prime subfield in infinite,
this implies the existence of $x\in \mathfrak s$ such that $\mathfrak s=\mathfrak a(x)$.
By \cite[Cor.\,3.8]{S75}, we then have $C_G(x)=C(\mathfrak s)$.
In view of \eqref{Ss}\;and\;\eqref{L},
this yields $L=C_G(x)$.
\end{proof}

\begin{proof}[Proof of  Theorems {\rm \ref{ratadj}}] Since the center of $G$ lies in the kernel of the adjoint representation
of $G$,
without changing the $G$-orbits
in $\mathfrak g$, we may (and shall)
assume that $G$ is a simply connected semisimple group.

Let $x$ be an element of $\mathfrak g$.
Our goal is to prove that
$G/C_G(x)$ is a rational variety. We may (and shall) assume that $x\neq 0$. We shall consider separately three cases:
\begin{enumerate}[\hskip 3.3mm\rm (i)]
\item
$x$ is nilpotent,
\item$x$ is semisimple,
\item
$x$ is neither nilpotent, nor semisimple.
\end{enumerate}

{\it Case {\rm(i)}.} Let $x$ be nilpotent.
Then by the Jacobson--Mo\-ro\-zov theo\-rem, there are
elements $h, y\in \mathfrak g$ such that $\{x, h, y\}$ is an ${\mathfrak s}{\mathfrak l}_2$-triple, i.e.,
$[h, x]=2x$, $[h, y]=-2y$, $[x, y]=h$. For every $i\in \mathbb Z$,\;put
\begin{equation*}
{\mathfrak g}(i):=\{z\in \mathfrak g\mid [h, z]=iz\}.
\end{equation*}
Then we have the decomposition $\mathfrak g=\bigoplus_{i\in\mathbb Z}{\mathfrak g}(i)$, which is a structure of a $\mathbb Z$-graded Lie algebra on $\mathfrak g$.

The subspace
$\mathfrak p:=\textstyle\bigoplus_{i\geqslant 0}\mathfrak g(i)$
is a parabolic subalgebra of $\mathfrak g$.
Let $P$ be the parabolic subgroup of $G$
such that
${\rm Lie}(P)=\mathfrak p$. Then
the $P$-stable subspace
$\mathfrak u:=\textstyle\bigoplus_{i> 0}\mathfrak g(i)$
is ${\rm Lie}({\rm Rad}_u (P))$.

We have $x\in {\mathfrak g}(2)\subseteq \mathfrak u$. By \cite[
Chap.\,III, Sect.\,4.20(i)]{SS70},
the $P$-orbit of $x$ is open in $\mathfrak u$, therefore
\begin{equation}\label{ratio}
\mbox{$P/C_P(x)$ is isomorphic to an open subset of an affine space.}
\end{equation}
By \cite[Chap.\,III, Sect.\,4.16]{SS70}, we have
$C_G(x)\subset P$; whence
\begin{equation}\label{=}
C_P(x)=C_G(x).
\end{equation}
In view of \eqref{=}, we have the following tower of algebraic groups:
\begin{equation}\label{tower}
G\supset P\supset C_G(x),
\end{equation}
By Example 2 and Lemma \ref{propcomp}(b)  applied to \eqref{tower}, we infer that
$G/C_G(x)$ contains an open subset isomorphic to
 $({\rm Rad}_u (P^{-}))\times (P/ C_G(x))$. Since the underlying variety of ${\rm Rad}_u (P^{-})$ is isomorphic to an affine space,
  we infer from \eqref{ratio} that $G/C_G(x)$ contains an open subset isomorphic to an open set of an affine space.
 Therefore, $G/C_G(x)$ is a rational variety.

\vskip 1mm

{\it Case {\rm (ii)}}. Let $x$ be semisimple. Then $C_G(x)$ is a Levi subgroup of a parabolic subgroup of $G$ in view of Theorem \ref{red-stab}.
Hence the variety $G/C_G(x)$ is rational by Example \ref{ex3}.

\vskip 1mm

{\it Case {\rm(iii)}.} Let $x$ be neither nilpotent, nor semisimple. Let $x=x_s+x_n$ be the Jordan decomposition of $x$.
By Theorem \ref{red-stab}, the group
$C_G(x_s)$ is a
Levi subgroup of a parabolic subgroup of $G$.
We have
$x_n\in {\rm Lie}(C_G(x_s))$ and it follows from the uniqueness of the Jordan decomposition that
\begin{equation}\label{zJ}
C_{G}(x)=C_{C_G(x_s)}(x_n).
\end{equation}
By Example \ref{ex3}, there is a two-term birational complement $S_1, S_2$ to $C_G(x_s)$ in $G$ such that
\begin{equation}\label{aass}
\mbox{$S_i$ is isomorphic to an affine space for every $i$.}
\end{equation}
Applying  Lemma \ref{propcomp} to $H=C_G(x_s)$, $Q=C_{C_G(x_s)}(x_n)$, we obtain from \eqref{zJ} that $G/C_{G}(x)$ contains an open set isomorphic
to $$S_1\times S_2\times (C_G(x_s)/C_{C_G(x_s)}(x_n)).$$ By (i), the variety $C_G(x_s)/C_{C_G(x_s)}(x_n)$
is rational. This and \eqref{aass} imply that  $G/C_{G}(x)$ is rational.
\end{proof}

\begin{remark} \label{rem}  Theorem \ref{ratadj} establishes a specific property of the ad\-joint representation: for some connected reductive groups $G$, there are finite-dimensional algebraic repre\-sentations not all of whose $G$-orbits are rational algebraic varie\-ties.

Indeed, in \cite[p.\,298]{Po11},
 \cite [Thm.\;2]{Po13}, \cite[Rem.\;1.5.9]{Po94} are con\-st\-ructed
connected reductive algeb\-raic groups $G$ with
 a finite subgroup $H$ such that the algebraic variety $G/H$ is nonrational.

 Being finite,
 $H$ is reductive; hence, by Matsu\-shima's cri\-terion,
the variety $G/H$ is affine. Therefore, by the embed\-ding theorem (see \cite[Thm.\;1.5]{PV94}), there exists a
$G$-equivariant (with respect to the natural action of $G$ on $G/H$) closed embedding of $G/H$ into some finite-dimensional algebraic $G$-module.
\end{remark}

\begin{proof}[Proof of  Theorems {\rm \ref{symplhom}}]\

 \eqref{or} By
\cite[Thm.\,5.4.1, Prop.\,5.1.1]{Ko70}, there are a uni\-que $G$-orbit $\mathcal O_X\subset  \mathfrak g$ and a unique morphism $\tau_X\colon X\to \mathcal O_X$ of Hamilto\-nian $G$-varieties. By \cite[Prop.\,5.1.1]{Ko70}, $\tau_X$ is a covering, and
for every $x\in X$, the identity component of the $G$-stabilizer $G_x$ of $x$ coincides with that of the $G$-stabilizer $G_{\tau_X(x)}$ of
$\tau_X(x)$.
Since $G$ acts on $X$ transitively and $X$ is affine,
$G_x$ is reductive by Matsushima's criterion. Therefore,  $G_{\tau_X(x)}$ is reductive as well. Then from the equivalence of  properties \eqref{ss}, \eqref{re}, and (e)
in Theorem \ref{red-stab}
we infer that
$\mathcal O_X\in \mathscr S$ and  $G_{\tau_X(x)}$
is connected. Since $G$ is simply connected, the latter implies that $\mathcal O_X$ is simply connected  as well. This, in turn,
implies that  $\tau_X$ is an isomorphism since $\tau_X$ is a covering.

\eqref{bije}
 In view of \eqref{or}, this follows from the fact that each $G$-orbit in $\mathfrak g$
is endowed with the standard structure of a Hamiltonian $G$-variety.

\eqref{scrv} This follows from \eqref{or}, since, as was proved above, ${\mathcal O}_X$  is a simply connected and, by Theorem \ref{ratadj}, rational variety.

\eqref{stab} This follows from \eqref{or} and the equivalence of \eqref{ss}
and (e)
in Theo\-rem \ref{red-stab}.

\eqref{Levs} This follows from the last statement of Theorem \ref{red-stab}.
\end{proof}

\end{document}